\documentclass[12]{amsart}

\usepackage{amssymb}
\newtheorem{theorem}{Theorem}

\newtheorem{lemma}[theorem]{Lemma}

\newtheorem{question}{Question}
\newtheorem{proposition}[theorem]{Proposition}

\newcommand{\id}{\mbox{id}}
\newcommand{\soc}{\operatorname{Soc}}
\newcommand{\Aut}{\operatorname{Aut}}
\newcommand{\sym}{\operatorname{Sym}}
\newcommand{\mpl}{\operatorname{mpl}}
\newcommand{\Ret}{\operatorname{Ret}}
\newcommand{\Ker}{\operatorname{Ker}}

\newcommand{\Lcal}{{\mathcal{L}}}

\begin{document}
\title[Nilpotent braces and the Yang--Baxter equation]{ On the Yang-Baxter equation and left  nilpotent left braces}
\author[Ced\'{o}, Gateva-Ivanova, Smoktunowicz]{Ferran Ced\'{o}, Tatiana
Gateva-Ivanova and Agata Smoktunowicz}

\subjclass[2010]{Primary 16N80, 16P90, 16N40 } \keywords{brace,
Yang-Baxter equation, nilpotent braces}

\date{\today}
\begin{abstract}
We study non-degenerate involutive set-theoretic solutions $(X,r)$
of the Yang-Baxter equation, we call them simply \emph{solutions}.
We show that the structure group $G(X,r)$ of a finite non-trivial solution $(X,r)$ cannot be an Engel group.
It is known that the structure group $G(X,r)$ of a finite
multipermutation solution $(X,r)$ is a poly-$\mathbb{Z}$ group, thus
our result gives a rich source of examples of braided groups and
left braces $G(X,r)$ which are poly-$\mathbb{Z}$ groups but not Engel groups.

We also show that a finite  solution of the Yang-Baxter equation
can be embedded in a convenient way into a finite brace and into a finite braided group.

For a left brace $A$, we explore the close relation between the
multipermutation level of the  solution associated with it and the radical chain $A^{(n+1)}=A^{(n)}* A$ introduced by Rump.
\end{abstract}

\maketitle

\section{Introduction}
Braces were introduced by Rump \cite{rump} to study non-degenerate
involutive set-theoretic solutions of the Yang-Baxter equation.

Recall that a left brace is a set $B$ with two operations, $+$ and
$\cdot$, such that $(B,+)$ is an abelian group, $(B,\cdot)$ is a
group and for every $a,b,c\in B$,
\begin{equation}\label{leftbrace}
a\cdot (b+c)+a=a\cdot b+a\cdot c.
\end{equation}
Right braces are defined similarly, changing the property
(\ref{leftbrace}) by $(a+b)\cdot c+c=a\cdot c+b\cdot c$. A two-sided
brace is a left brace which is also  a right brace. In any left
brace $(B,+,\cdot)$ one defines another operation $*$ by the rule
$$a*b=a\cdot b-a-b,
$$
for $a,b\in B$. It is known that $(B,+,\cdot)$ is a two-sided brace
if and only if $(B,+,*)$ is a Jacobson radical ring. Conversely, if
$R$ is a Jacobson radical ring, then one defines a new operation
$\circ$ on $R$ by $a\circ b=ab+a+b$ and $(R,\circ)$ is called the
adjoint group of the radical ring $R$. Then $(R,+,\circ)$ is a
two-sided brace. Hence the study of two-sided braces is equivalent
to the study of Jacobson radical rings.

%\textbf{By convention,  when we study a Jacobson radical ring $R$,
%we shall denote the operation multiplication, by "$*$", and will
%write $R = (R, +, *)$, so the associated two-sided brace is simply
%$(R, +, .)$, where, as usual, $ab = a*b+a +b.$}

In general the operation $*$ in a left brace $B$ is not associative,
but it is left distributive with respect to the sum, that is
$$a*(b+c)=a*b+a*c,$$
for $a,b,c\in B$.

Let $B$ be a left brace. For $a\in B$,  let  $\mathcal{L}_a\colon
B\longrightarrow B$ be the map defined by $\mathcal{L}_a(b)=ab-a$
for all $b\in B$. It is known that $\mathcal{L}_a$ is an
automorphism of the additive group of the left brace $B$ and the
 map $\mathcal{L}\colon (B,\cdot)\longrightarrow \Aut(B,+)$, defined by
$a\mapsto \mathcal{L}_a$, is a morphism of groups.  The kernel of
this morphism is called the socle of $B$,
$$\soc(B):=\{ a\in B\mid \mathcal{L}_a=\id \}=\{ a\in B\mid ab=a+b,\; \mbox{ for all }b\in B \}.$$
In fact the socle of a left brace $B$ is an ideal of $B$, that is, a
normal subgroup of its multiplicative group invariant by the maps
$\mathcal{L}_a$ for all $a\in B$. In particular, $\soc(B)$ is also a
subgroup of the additive group of $B$. Note that if $a,b\in
\soc(B)$, then $a-b=\mathcal{L}_b(b^{-1}a)\in \soc(B)$. Therefore
the quotient of the multiplicative group $B/\soc(B)$ is also the
quotient of the additive group and $(B/\soc(B),+,\cdot)$ is a left
brace, the left brace quotient of $B$ modulo its ideal $\soc(B)$.

Let $X$ be a non-empty set. Recall that a map $r\colon X\times
X\longrightarrow X\times X$ is a set-theoretic solution of the
Yang-Baxter equation if
$$r_{12}r_{23}r_{12}=r_{23}r_{12}r_{23},$$
where $r_{12},r_{23}\colon X\times X\times X\longrightarrow X\times
X\times X$ are the maps $r_{12}=r\times \id_X$ and
$r_{23}=\id_X\times r$. We will write $r(x,y)=({}^{x}y, x^{y})$. The
map $r$ is \emph{non-degenerate} if for every $x\in X$ the maps
$y\mapsto {}^{x}y$ and $y\mapsto y^{x}$ are bijective, $r$ is
\emph{involutive} if $r^2=\id_{X^2}$.
\smallskip

\noindent {\bf Convention.} By \emph{a solution of the YBE} (or
shortly, \emph{a solution}) we mean a non-degenerate involutive
set-theoretic solution $(X,r)$ of the Yang-Baxter equation.
\smallskip

Let $(X,r)$ be a solution of the YBE. The structure group of
$(X,r)$ is the group $G(X,r)$ with presentation
$$G(X,r)=\langle X\mid xy=({}^{x}y)(x^{y}),\; x,y\in X\rangle.$$
(Some authors call $G(X,r)$ the YB group of $(X,r)$).

It follows from the results in \cite{ess} that $X$ is naturally
embedded in $G(X,r)$. One can define a sum on $G(X,r)$ such that
$(G(X,r),+)$ is a free abelian group with basis $X$. Moreover,
$(G(X,r),+,\cdot)$ is a left brace such that
${}^{x}y=\mathcal{L}_x(y)\in X$ for all $x,y\in X$, see \cite{cjo2,
Tatyana}.

We say that this is \emph{the canonical left brace structure on}
$G(X,r)$. It is known that the group $G$ acts on the set $X$ from
the left (and from the right). Moreover, the assignments $x \mapsto
\Lcal_x$ extend to a group homomorphism $\Lcal: G \longrightarrow
\sym_X$. The image $ \Lcal (G)$ of this homomorphism is a subgroup
of $\sym_X$ called \emph{the permutation group  of} $(X,r)$  and
denoted by $\mathcal{G}(X,r).$ It is known that
$\mathcal{G}(X,r):=\langle \mathcal{L}_x \mid x\in X \rangle$, where
$\Lcal_x(y)={}^{x}y$, for all $x,y\in X$. The group epimorphism
$\Lcal: G(X,r)\longrightarrow \mathcal{G}(X,r), x\mapsto \Lcal_x$
has kernel $\Ker \Lcal = \soc(G(X,r))$ (as sets).
 Thus $\mathcal{G}(X,r)$ inherits a
structure of a left brace via this natural isomorphism of groups, we
say that this is \emph{the canonical structure of a left brace on}
$\mathcal{G}(X,r)$. Moreover, $G(X,r)/\soc(G(X,r))\cong
\mathcal{G}(X,r)$ as symmetric groups (i.e. involutive braided
groups) and as left braces, \cite{ Tatyana, cjo2}.

%\emph{ ****Another group associated with a solution $(X,r)$ of the
%YBE is its permutation group  It is known that the map
%$X\longrightarrow \mathcal{G}(X,r)$ defined by $x\mapsto \Lcal_x$
%can be extended to a homomorphism of groups $G(X,r)\longrightarrow
%\mathcal{G}(X,r)$ with kernel equal to $\soc(G(X,r))$. Therefore
%$G(X,r)/\soc(G(X,r))\cong \mathcal{G}(X,r)$ as multiplicative
%groups. Thus $\mathcal{G}(X,r)$ inherits a structure of left brace
%via this natural isomorphism of groups, we say that this is the
%natural structure of left brace of $\mathcal{G}(X,r)$.****}

In this paper, we prove some general results about braces and apply
these to study the close relations between the properties of
solutions $(X,r)$ and their associated left braces $G(X,r)$. This is
in the spirit of \cite{Tatyana}  and \cite{cjo2}.

%***\emph{In this paper, we study the relation between some
%properties of the left brace $G(X,r)$ and some properties of the
%solution $(X,r)$ of the YBE.}***

\section{Some results on  $G(X,r)$}
The results of this section were motivated by a result from
\cite{Tatyana}, which assures that the structure group $G(X,r)$ of a
non-trivial solution $(X,r)$ of the Yang Baxter cannot be a
two-sided brace.

Let $(B,+,\cdot)$ be a left brace. As usual, for any $a,b\in B$ and
positive integer $m$, $ab$ will denote $a\cdot b$ and $a^m$ will
 denote $a\cdot a\cdots a$ (where $a$ appears $m$ times).

\begin{lemma}\label{nil} Let $B$ be a left brace whose additive group
 $(B,+)$ is torsion-free.
 Assume that
 $a, b\in B$ and that there is an integer
 %natural number
 $n(a,b)$ such that $a*
(a*(\ldots a*(a* b)\ldots ))=0$ (where $a$ occurs $n(a,b)$ times
and $b$ once in this equation).
 Assume moreover that $\mathcal{L}_{a^n}=\id$ for some integer
 $n$.
  Then $a*b=0$ or equivalently,  $a\cdot b=a+b$.
\end{lemma}

%\begin{lemma}\label{nil} {\bf Let $B$ be a left brace whose additive group
% $(B,+)$ is torsion-free.
% Assume that
% $a, b\in B$ and that there is a natural number $n(a,b)$ such that $a*
%(a*(\ldots a*(a* b)\ldots ))=0$ (where $a$ occurs $n(a,b)$ times
%and $b$ once in this equation).
% Assume moreover that at least one of the following assertions holds:
%\begin{itemize}
%\item [1.] $[B: \soc(B)]=n<\infty$.
%\item [2.] $\mathcal{L}_{a^n}=\id$ for some $n$.
%\end{itemize}
 % Then $a*b=0$ or equivalently,  $a\cdot b=a+b$.}
%\end{lemma}
\begin{proof}
%Since $[B: \soc(B)]=n<\infty$, we have that $\mathcal{L}_{a^n}=\id$.
Note that $\mathcal{L}_a(b)=a*b+b$, for $a,b\in B$. Let $m$ be a
positive integer. Let $e_1(a,b)=a*b$ and $e_{m+1}(a,b)=a*e_m(a,b)$,
for $a,b\in B$. It can be proved by induction on $m$ that
$$\mathcal{L}_{a^m}(b)=b+\sum_{i=1}^m {m\choose i} e_{i}(a,b).$$
 Since $\mathcal{L}_{a^n}=\id$, we have
$$b=\mathcal{L}_{a^n}(b)=b+\sum_{i=1}^n {n\choose i} e_{i}(a,b),$$
and thus
$$ne_1(a,b)=-\sum_{i=2}^n {n\choose i} e_{i}(a,b).$$
Hence
\begin{eqnarray}\label{eq1}
&&ne_{k}(a,b)=-\sum_{i=2}^n {n\choose i}
e_{i+k-1}(a,b),\end{eqnarray}
for all positive integer $k$.

Suppose that  $e_1(a,b)=a*b\neq 0$.
Let $n(a,b)$ be the smallest positive integer such that
$e_{n(a,b)}(a,b)=0$. Then, by (\ref{eq1}),
$$ne_{n(a,b)-1}(a,b)=-\sum_{i=2}^n {n\choose i} e_{i+n(a,b)-2}(a,b)=0.$$
Since $(B,+)$ is torsion-free, we have that $e_{n(a,b)-1}(a,b)=0$,
in contradiction with the definition of $n(a,b)$.
\end{proof}
%\textbf{Clearly, every $a\in G$ satisfies $\mathcal{L}_{a^n}=\id$,
%whenever $[B: \soc(B)]=n<\infty$.}

Let $G$ be a group.   Following the notation of \cite{robinson}, for
$g,h\in G$, we denote by $[g,h]$ the element $[g,h]=g^{-1}h^{-1}gh$.
Recall that the group $G$ is an Engel group if and only if for each
$g,h\in G$ there exists a positive integer $n(g,h)$ such that
$[[\dots [[g,h],h]\dots],h]=1$, where $h$ occurs $n(g,h)$ times.

\begin{theorem}\label{engel} Let $B$ be a left brace such that its additive group
$(B,+)$ is torsion-free and $[B: \soc(B)]=n<\infty$. If the
multiplicative group $(B,\cdot)$ of the left brace $B$ is an Engel
group, then $B$ is a trivial brace, that is $a\cdot b=a+b$, for
all $a,b\in B$.
\end{theorem}

\begin{proof}
Let $a\in B$ and $c\in \soc(B)$. Note that
$$\mathcal{L}_a(c)=ac-a=aca^{-1}a-a=aca^{-1}+a-a=aca^{-1}.$$
Hence
\begin{eqnarray*}[a,c]&=&a^{-1}c^{-1}ac=a^{-1}c^{-1}a+c\\
&=&\mathcal{L}_{a^{-1}}(c^{-1})+c= \mathcal{L}_{a^{-1}}(-c)+c\\
&=&-\mathcal{L}_{a^{-1}}(c)+c= -a^{-1}*c-c+c\\
&=&-a^{-1}*c=(a^{-1}*c)^{-1}.
\end{eqnarray*}
Hence $[c,a]=a^{-1}*c$. Therefore $[[\dots
[[c,a],a]\dots],a]=a^{-1}*(\dots a^{-1}*(a^{-1}*c)\dots )$. Let
$b\in B$. Since $[B: \soc(B)]=n<\infty$ and $\soc(B)$ is an ideal of
$B$, $nb\in\soc(B)$. Since $(B,\cdot)$ is an Engel group, there
exists a positive integer $m$ such that $$[[\dots
[[(nb),a],a]\dots],a]=1,
$$ (where $a$ occurs $m$ times).
Hence $$0=1=a^{-1}*(\dots a^{-1}*(a^{-1}*(nb))\dots
)=n(a^{-1}*(\dots a^{-1}*(a^{-1}*b)\dots )),$$ (where $a^{-1}$
appears $m$ times). But $(B,+)$ is torsion free, hence
$$a^{-1}*(\dots a^{-1}*(a^{-1}*b)\dots )=0,$$ where $a^{-1}$ occurs $m$ times.

By Lemma~\ref{nil}, $a^{-1}*b= 0$. Therefore $a*b = 0$, for all $a,b
\in B,$  or equivalently, $B$ is a trivial left brace.
\end{proof}

We call a left brace $B$  left nilpotent if $B^{n}=0$ for some $n$,
where $B^{n+1}=B* B^{n}$ is the chain introduced by Rump in
\cite{rump}. As a consequence of Lemma~\ref{nil} and
Theorem~\ref{engel}, we have the following two results.

%\begin{theorem}
%%%%\label{one}
% Let $(X,r)$ be a finite solution of the YBE. Assume that for each  $a, b\in X$
% there is
% a natural number
% an integer $n=n(a,b)$ such that, in $G(X,r)$,
% we have  $a* (a*(\ldots a*(a* b)))=0$ (where $a$ occurs $n$ times and $b$ occurs once in
% this equality
% %equation),
% then $(X,r)$ is the trivial solution.
%In particular, if  $G(X,r)$ is a left nilpotent left brace then
%$(X,r)$ is the trivial solution.
%\end{theorem}

\begin{theorem} \label{one}
Let $(X,r)$ be a finite solution of the YBE. Assume that for each
$a, b\in X$ there is a positive integer $n=n(a,b)$ such that
the equality $a* (a*(\ldots a*(a* b)))=0$ holds in $G(X,r)$,($a$
occurs $n$ times and $b$ occurs once in this equality). Then $(X,r)$
is the trivial solution. In particular, if  $G(X,r)$ is a left
nilpotent left brace, then $(X,r)$ is the trivial solution.
\end{theorem}
\begin{proof}
 Since $G(X,r)/\soc(G(X,r))\cong \mathcal{G}(X,r)$ is a
subgroup of the symmetric group $\sym_X$ of the finite set $X$, we
have that $[G(X,r): \soc(G(X,r))]<\infty$. Hence,  by
Lemma~\ref{nil}, ${}^{a}b=\mathcal{L}_a(b)=ab-a=a+b-a=b$, for all
$a,b\in X$. In particular, $(X,r)$ is the trivial solution.
\end{proof}

%
%\emph{***At the beginning of \cite[Section 4]{jo} one can read
%that it is known that any ordered abelian-by-finite group is
%abelian. Also it is known that any torsion-free nilpotent group is
%ordered (see \cite[Lemma 13.1.6]{passman}). It is known that if
%$(X,r)$ is a solution of the YBE, then $G(X,r)$ is a torsion-free,
%solvable and abelian-by-finite group. Therefore, if $G(X,r)$ is
%nilpotent, then it is abelian (in this case the left brace
%structure of $G(X,r)$ is trivial and $(X,r)$ is the trivial
%solution). We have the following related result.***I suggest we
%replace the para marked with *** with the following para written
%in bold****}

It is known that any ordered abelian-by-finite group is abelian, see
for example \cite[Section 4]{jo}. It is also known that any
torsion-free nilpotent group is ordered (see \cite[Lemma
13.1.6]{passman}). Recall  that if $(X,r)$ is a finite
solution of the YBE,  then $G(X,r)$ is a torsion-free, solvable
and abelian-by-finite group (see \cite{ess} and \cite{GIVdB}).
Therefore, if $G(X,r)$ is nilpotent, then it is abelian. In this
case the canonical left brace structure on $G(X,r)$ is trivial and
$(X,r)$ is the trivial solution. We have the following related
result.

\begin{theorem}
\label{two}
 Let $(X,r)$ be a finite solution of the YBE.
  If the structure group
$G(X,r)$ is an Engel group, then $(X,r)$ is the trivial solution.
\end{theorem}
\begin{proof}
This is a consequence of Theorem~\ref{engel}.
\end{proof}

\section{Right nilpotent left braces}

Etingof, Schedler and Soloviev in \cite{ess} introduced the retract
solution of a given solution of the YBE. Let $(X,r)$ be a solution
of the YBE.  The retract relation $\sim $ on the set $X$ with
respect to $r$ is defined by $x\sim y$ if $\sigma_{x}=\sigma_{y}$,
where $\sigma_x(z)={}^{x}z$. Then the retraction of $(X,r)$ is
$\Ret(X,r)=([X], r_{[X]})$, where $[X]=X/\sim$ and
$$r_{[X]}([x],[y])=([{}^{x}y],[x^{y}]),$$
where $[x]$ denotes the $\sim$-class of $x\in X$. We define
$\Ret^1(X,r)=\Ret(X,r)$ and $\Ret^{k}(X,r)=\Ret(\Ret^{k-1}(X,r))$
for $k>1$. A solution $(X,r)$ of the YBE is called a
multipermutation solution of level $m$ if $m$ is the smallest
nonnegative integer such that the solution $\Ret^{m}(X,r)$ has
cardinality $1$; in this case we write $\mpl(X,r)=m$.

Let $B$ be a left brace. By $B^{(m)}$ we mean the chain of ideals
introduced by Rump in \cite{rump}, so $B^{(1)}=B$ and
$B^{(n+1)}=B^{(n)}*B$. We say that $B$ is right nilpotent if there
exists a positive integer $n$ such that $B^{(n)}=0$.

Recall that if $B$ is a left brace, then the map $r\colon B\times
B\longrightarrow B\times B$ defined by
$$r(a,b)=(\mathcal{L}_a(b),\mathcal{L}^{-1}_{\mathcal{L}_a(b)}(a)),$$
is a solution of the YBE. This is the solution of the YBE associated
with the left brace $B$ (see \cite{cjo}).

\begin{proposition}\label{five}
Let $B$ be a nonzero left brace and let $(B,r)$ be its
associated solution of the YBE. Then the multipermutation level of
$(B,r)=m<\infty $ if and only if $B^{(m+1)}=0$ and $B^{(m)}\neq 0.$
\end{proposition}
\begin{proof} Note that
$\soc(B)=\{b\in B\mid  b* a=0$ for every $a\in B\}$.

First we shall prove the implication ($\mpl (B,r) = m) \Rightarrow
(B^{(m+1)}=0$ and $B^{(m)}\neq 0$). We use induction on $m = \mpl
(B,r)$. Suppose $\mpl (B,r) = 1$. Therefore,
$\mathcal{L}_{a}(b)=a*b+b=b$ which is equivalent to   $a*b=0$ for
all $a,b\in B$. It follows that $B*B=0$, so $B^{(2)}=0$.
 But $B$ is a nonzero left brace, hence $B^{(1)}=B\neq 0$. This gives the base for induction.

Suppose now that for all $k$, $1\leq k\leq m-1$, the condition
$\mpl (B,r) = k \leq  m-1$ implies $B^{(k+1)} = 0$ and
$B^{(k)}\neq 0$. Assume that  $\mpl(B,r) = m$, then the retraction
$\Ret(B,r) = ([B],r_{[B]})$ has multipermutation level $m-1$.
%\emph{***Moreover,  $[B]$ is a left brace isomorphic to the left
%brace $B/\soc(B)$ and $\Ret(B,r)$ is isomorphic to the solution of
%the YBE associated with $B/\soc(B)$ (\cite{rump}, \cite{cjo},
%\cite{Tatyana}).***}

Moreover, there is an isomorphism of left braces (or equivalently an
isomorphism of braided groups) $B/\soc(B) \cong [B]$ and $\Ret(B,r)$
is isomorphic to the solution of the YBE associated with $B/\soc(B)$
(\cite{rump}, \cite{cjo}, \cite{Tatyana}). Hence by the inductive
assumption $(B/\soc(B))^{(m)}= 0$ and $(B/\soc(B))^{(m-1)}\neq 0$.
This implies
 $B^{(m)}\subseteq \soc(B)$ and that $B^{(m-1)}$ is not a subset of $\soc(B)$. Therefore $B^{(m+1)}=0$ and $B^{(m)}\neq 0$.

 Now we prove the inverse implication:
($B^{(m+1)}=0$ and $B^{(m)}\neq 0$) $\Rightarrow$ ($\mpl (B,r) =
m$).

The base for the induction is clear. Assume that for all $k\leq m$
the implication is true. Suppose that $B$ is a left brace such that
$B^{(m+2)}=0$ and $B^{(m+1)}\neq
 0$.
Recall  that $B^{(m+2)}=B^{(m+1)}*B$, therefore
$(B/\soc(B))^{(m+1)}=0.$ On the other hand  $B^{(m+1)}\neq 0$ and
$B^{(m+1)}=B^{(m)}*B$ imply $(B/\soc(B))^{(m)}\neq 0$. By the
inductive assumption
 $\mpl (\Ret(B,r))=m$, and therefore, $\mpl(B,r) = m+1$. This proves the
proposition.
\end{proof}

\section{Embedding solutions and groups into finite braces and finite rings}

%\emph{***In this section we will show that a finite solution of the
%Yang-Baxter equation can  be embedded into a finite left brace.
%Recall that it was shown in \cite{Tatyana} that there is a
%one-to-one correspondence between left braces and braided groups
%with involutive braiding operators (so called symmetric groups in
%the sense of Takeuchi). Therefore Proposition \ref{three}
% also shows that a finite solution of the Yang-Baxter equation can be
%embedded (in a convenient way) into a braided group with an
%involutive braiding operator.***}
%

In this section we will show that a finite solution of  the YBE
can  be embedded (in an explicit way) into a finite left brace.
Recall that it was shown in \cite{Tatyana} that there is a canonical
one-to-one correspondence between left braces and symmetric groups
(in the sense of Takeuchi \cite{Ta}). Therefore Proposition
\ref{three}
 also shows explicitly how to embed a finite solution of the YBE into a finite symmetric group.

%\begin{proposition}\label{three}
%Let $(X,r)$ be a finite solution of the YBE. Then there is a
%finite left brace $(G, +, \cdot)$ such that $X\subseteq G$
%generates the additive group of $G$ and $(X,r)$ is a subsolution
%of the solution of the YBE associated with the left brace $G$,
%that is, for every $a,b\in X$,  ${ }^ab=\mathcal{L}_a(b)=ab-a$ and
%$a^{b}=\mathcal{L}^{-1}_{\mathcal{L}_b(a)}(b)$.
%\end{proposition}

\begin{proposition}\label{three}
Let $(X,r)$ be a finite solution of the YBE. Then there is a finite
left brace $(B, +, \cdot)$ such that $X\subseteq B$ generates the
additive group of \textbf{$B$}. Moreover, $(X,r)$ is a subsolution
of the solution $(B, \sigma)$ associated canonically with the left
brace $B$, that is, $X$ is $\sigma$-invariant and $r =
\sigma_{|X\times X}$ is the restriction of $\sigma$ on $X \times
X$.\end{proposition}
\begin{proof}
Let $G=G(X,r)$ be the structure group of the finite solution
$(X,r)$. We know that the additive group of the left brace $(G, +,
.)$ associated with $G$ is free abelian with basis $X$ and ${
}^ab=\mathcal{L}_a(b)=ab-a$ and
$a^{b}=\mathcal{L}^{-1}_{\mathcal{L}_a(b)}(a)$, for all $a,b\in X$.
Since $X$ is finite, $[G: \soc(G)]<\infty$, say $[G: \soc(G)]=n$.
Consider the set $I=\{ ng\mid g\in G(X,r)\}$. We claim that $I$ is
an ideal of the brace $G$, that is, $I$ is a normal subgroup of the
multiplicative group $(G, .)$ which is invariant with respect to the
left actions  by elements of $G$, \cite{cjo2}. It is clear that $I$
is an additive subgroup of $(G, +)$ and $I\subseteq \soc(G(X,r))$.
Then for $u,v \in I$ one has $uv = u+ {}^uv = u+v \in I,$
$u^{-1}=-u\in I$, so $I$ is a subgroup of $G$. Let $g,h\in G$. Then
\[h(ng)h^{-1}= ({}^h(ng))(h^{ng})(h^{-1})=  ({}^h(ng))(hh^{-1})= {}^h(ng)= n({}hg)\in I.\]
Thus, $I$ is a normal subgroup of $(G, .)$ which is also invariant
under the left action by elements of $G$. Therefore $I$ is an ideal
of the left brace $(G, +, .)$.  It is not difficult to show that
brace quotient $B= G/I$ is a finite left brace of order $n^{m}$,
 where $m$ is the cardinality of $X$.
 Observe
also that for any two elements $x, y \in X, x\neq y$, one has
$x-y\notin I$, since the additive group of $(G, +)$ is free abelian
with a basis $X$. Now the restriction of the natural map
$G\rightarrow G/I= B$ on the set $X$ is injective. The proposition
has been proved. \end{proof}

 At a conference in Porto Cesareo, B. Amberg mentioned
that he and his collaborators  first  became interested in
Jacobson radical rings because they gave them a way to construct
examples of triply-factorizable groups. Later, they found more
ways of constructing such examples. Triply factorized groups can
be also used to define braces; see \cite[Theorem~18]{sysak2}.
Interesting results on triply factorized  groups can be found in
\cite{104, 101, 4, 6, sysak2, 102}. Triply factorized groups are
for example useful for investigating the structure of normal
subgroups of a group $G=AB$ which is a product of two subgroups.
Several authors investigated connections between triply factorized
groups and nearrings \cite{sysak2}, \cite{102}. It might be
interesting to investigate the connections  between nearrings and
braces. We would like to pose a related open question:

\begin{question}
 Investigate whether there is any relation between nearrings and
solutions of the YBE?
\end{question}

The multiplicative group of a brace $A$ is also called an adjoint
group of brace $A$. Observe that \cite[Corollary~3.6]{cjr}
asserts that every finite solvable group is a subgroup of an
adjoint group of some left brace. We also make the following simple
remark which follows from  \cite[Lemma~8.1]{cjo2} and
\cite[Corollary~3.8]{cjr}.
\smallskip

\noindent {\bf Remark.} (Related to \cite[Corollary~3.8]{cjr} and
\cite[Lemma~8.1]{cjo}) Every finite nilpotent group is a subgroup of
the adjoint group of a finite nilpotent ring.
\smallskip

  Let $p$ be a prime. By \cite[Lemma~8.1]{cjo}, every finite
$p$-group is isomorphic to a subgroup of the adjoint group of a
finite nilpotent ring $R$ such that $R$ has cardinality a power of
$p$. Let $G$ be a finite nilpotent group. Let $p_1,\dots p_m$ be the
the distinct prime divisors of the order of $G$. Let $P_i$ be the
Sylow $p_i$-subgroup of $G$. Then $P_i$ is isomorphic to a subgroup
of the adjoint group of a finite nilpotent ring $R_i$. Since $G\cong
P_1\times \cdots \times P_m$, it is clear that $G$ is isomorphic to
a subgroup of the adjoint group of the finite nilpotent ring
$R_1\times \cdots\times R_m$.
If $R$ is a ring, then the adjoint semigroup of $R$ is defined by
$a\circ b=a+b+a\cdot b$.

  By following the technique of the
proof of \cite[Lemma~8.1]{cjo} we get the following result.

\begin{proposition}
Let $G$ be a group and let $R$ be a ring (with unit). Then $G$ is
isomorphic to a subgroup of the adjoint semigroup of the group ring
$R[G]$.
\end{proposition}

\begin{proof}
Let $f\colon G\longrightarrow R[G]$ be the map defined by
$f(g)=g-1$, for $g\in G$. Clearly $f$ is injective. Let $g,h\in G$.
We have that
$$f(gh)=gh-1=(g-1)(h-1)+g-1+h-1=(g-1)\circ (h-1)=f(g)\circ f(h).$$
Therefore $f$ is an injective homomorphism of semigroups from $G$
into the adjoint semigroup of $R[G]$.
\end{proof}

%We ask the following question.
%\begin{question}
% Is every finite group a subgroup of the adjoint semigroup of some finite (associative) ring?
%\end{question}

\section*{Acknowledgments}
The research of the first author was partially supported by grants
DGI MICIIN  MTM2011-28992-C02-01 and MINECO MTM2014-53644-P. The
research of the second author was partially supported by Grant I
02/18 "Computational and Combinatorial Methods in Algebra and
Applications"
 of the Bulgarian National Science Fund,  by The Abdus Salam
International Centre for Theoretical Physics (ICTP), Trieste,  and
by Max-Planck Institute for Mathematics, Bonn.
 The research of the third author was supported with ERC grant 320974.

$ $

$ $

{\small Ferran Ced{\' o}, Departament de Matem\`{a}tiques,
Universitat Aut\`{o}noma de Barcelona, 08193 Bellaterra (Barcelona),
Spain
\\

Tatiana Gateva-Ivanova,  American University in Bulgaria, 2700
Blagoevgrad,  and Institute of Mathematics and Informatics,
Bulgarian Academy of
Sciences, 1113 Sofia, Bulgaria \\

Agata Smoktunowicz, School of Mathematics,
The University of Edinburgh,
James Clerk Maxwell Building,
The Kings Buildings, Mayfield Road
EH9 3JZ, Edinburgh}
\\

E-mail: \email{ cedo@mat.uab.cat, Tatyana@aubg.edu, A.Smoktunowicz@ed.ac.uk}

\end{document}